\documentclass[12pt]{amsproc}
\usepackage{amssymb}
\usepackage{amsmath}
\usepackage{amsfonts}
\usepackage{geometry}
\usepackage{graphicx}
\providecommand{\U}[1]{\protect\rule{.1in}{.1in}}
\theoremstyle{plain}

\newtheorem{corollary}{Corollary}

\newtheorem{lemma}{Lemma}

\newtheorem{proposition}{Proposition}

\newtheorem{theorem}{Theorem}
\numberwithin{equation}{section}

\begin{document}
\title[Doubling property and vanishing order of Steklov eigenfunctions]{Doubling property and vanishing order of Steklov eigenfunctions}
\author{ Jiuyi Zhu}
\address{
 Department of Mathematics\\
Johns Hopkins University\\
Baltimore, MD 21218, USA\\
Emails:  jzhu43@math.jhu.edu }
\thanks{\noindent }
\date{}
\dedicatory{ }

\begin{abstract}
The paper is concerned with the doubling estimates and vanishing
order of the Steklov eigenfunction on the boundary of a smooth
boundary domain $\mathbb R^n$. The eigenfunction is given by a
Dirichlet-to-Neumann map. We improve the doubling property shown by
Lin and Bellova \cite{BL}. Furthermore, we show that the vanishing
order of Steklov eigenfunction is everywhere less than $C\lambda$
where $\lambda$ is the Steklov eigenvalue and $C$ depends only on
$\Omega$.
\end{abstract}

\maketitle
\section{Introduction}
In this paper, we study the doubling estimates and vanishing order
of Steklov eigenfunction on the boundary of a smooth bounded domain
in $\mathbb R^n$. The Steklov eigenvalue problem is given by
\begin{equation}
\left \{ \begin{array}{lll} \triangle u=0  &\quad \mbox{in}\ \Omega, \medskip\\
\frac{\partial u}{\partial \nu}=\lambda u &\quad \mbox{on}\
\partial\Omega,
\end{array}
\right. \label{ste}
\end{equation}
where $\nu$ is the exterior unit normal and $\Omega\subset \mathbb
R^n$ is a bounded smooth domain. One can also consider the Steklov
eigenvalue problem on any smooth Riemannian manifold. It is known
that the model (\ref{ste}) has a discrete spectrum
$\{\lambda_j\}^\infty_{j=0}$ with $$0=\lambda_0<\lambda_1\leq
\lambda_2\leq \lambda_3,\cdots, \
 \ \mbox{and} \ \ \lim_{j\to\infty}\lambda_j=\infty.$$ The Steklov eigenfunctions were
introduced by Steklov \cite{St} in 1902. It interprets the vibration
of a free membrane with uniformly distributed mass on the boundary.
The model (\ref{ste}) also plays a important role in conductivity
and harmonic analysis as it was initially studied by Calder\'{o}n
\cite{C}, since the set of eigenvalues for the Steklov problem is
the same as the set of eigenvalues of the Dirichlet-to-Neumann map.
The Steklov eigenfunctions and eigenvalues have appeared in quite a
few physical fields, such as fluid mechanics, electromagnetism,
elasticity, etc. For instance, It is found applications in the
investigation of surface waves \cite{BS}, the stability analysis  of
mechanical oscillators immersed in a viscous fluid \cite{CPV}, and
the study of the vibration modes of a structure in contact with an
incompressible fluid \cite{BRS}.

Vanishing order of solution is a quantitative behavior of the strong
unique continuation property. The vanishing order of function $u$ at
$x_0$ is $l$ if  $D^{\alpha}u(x_0)=0$ for all $|\alpha|\leq l$ in
Taylor expansion. In order to obtain the order of vanishing for the
Steklov eigenfunction, we usually need to show the doubling
estimates for (\ref{ste}). Doubling estimates are also crucial tools
in proving strong unique continuation and obtaining the measure of
nodal set of eigenfunctions. It describes that the $L^2$ norm of
solutions in a ball $\mathbb B_{2r}$ is controlled by the $L^2$ norm
of solutions in a small ball $\mathbb B_{r}$. For the eigenfunction
of the Laplace-Beltrami operator $\triangle_g$ on the smooth
Riemannian manifold $\mathcal {M}$, i.e.
\begin{equation}-\triangle_g u+\lambda u=0 \quad \mbox{on} \
\mathcal {M}, \label{eig}
\end{equation}
Donnelly and Fefferman \cite{DF} showed that the vanishing order of
$u$ is bounded everywhere in $\mathcal {M}$ by $C\sqrt{\lambda}$,
where $C$ depends only on the manifold $\mathcal {M}$ and $\lambda$
is large. Their proof relies on the following doubling estimates
\begin{equation} \int_{\mathbb B_{2r}(x)}u^2\leq e^{C\sqrt{\lambda}}
\int_{\mathbb B_{r}(x)}u^2 \label{dou}
\end{equation} for any $x\in  \mathcal {M}$. With the aid of
(\ref{dou}), Donnelly and Fefferman \cite{DF} obtained the bound for
the nodal set of eigenfunctions on an analytic manifold, that is,
$$ C\sqrt{\lambda}\leq H^{n-1}(\{x\in \mathcal {M}, u(x)=0\})\leq
C\sqrt{\lambda}$$ with $C$ depending only on $\mathcal {M}.$

The study of doubling estimates and nodal sets of Steklov
eigenfunctions for large $\lambda$ was first initiated by  Lin and
Bellova in \cite{BL}. They were able to prove that
\begin{equation}\int_{\mathbb B_{2r}(x)\cap \partial \Omega}u^2\leq e^{C\lambda^5}
\int_{\mathbb B_{r}(x)\cap\partial \Omega}u^2 \label{dou2}
\end{equation}
for $x\in\partial\Omega$ and $r\leq {r_0}/{\lambda}$ with $r_0$, $C$
depending only on $\Omega$ and $n$. Furthermore, they obtained the
upper bound for the nodal set of Steklov eigenfunctions on the
analytic domain. It was shown in \cite{BL} that
$$ H^{n-2}(\{x\in\partial\Omega, u(x)=0\})\leq
C\lambda^6.$$  Later on, Zelditch \cite{Z} was able to improve the
upper bound and show that optimal upper bound for the nodal set is
$C\lambda$ on analytic domains.  Zelditch's proof does not rely on
doubling estimates. It is left as an interesting problem to study
the doubling property of Steklov eigenfunctions. Our main goal is to
improve the doubling estimates in (\ref{dou2}) on smooth domains.
\begin{theorem}
Let $\Omega\in \mathbb R^n$ be a smooth domain. There exist positive
constants $r_0$, $C$ depending only on $\Omega$ such that
\begin{equation}\int_{\mathbb B_{2r}(x)\cap \partial \Omega}u^2\leq e^{C\lambda}
\int_{\mathbb B_{r}(x)\cap\partial \Omega}u^2
\end{equation}
holds with $r\leq {r_0}/{\lambda}$ and $x\in\partial\Omega$.
\label{th1}
\end{theorem}
An immediate consequence of the integral form of doubling estimates
is the doubling property in $L^\infty$ norm. We were able to show
the following corollary.
\begin{corollary}
There exist positive constants $r_0$, $C$ depending only on $\Omega$
and $n$ such that for $r\leq {r_0}/{\lambda}$ and
$x\in\partial\Omega$, there holds
\begin{equation}
\|u\|_{L^\infty(\mathbb
B_{\frac{2r}{\lambda}}(x)\cap\partial\Omega)}\leq
e^{C{\lambda}}\|u\|_{L^\infty(\mathbb
B_{\frac{r}{\lambda}}(x)\cap\partial\Omega)}.
\end{equation}
\label{cor1}
\end{corollary}

 Based on Carleman estimates and doubling property, we are able to obtain the vanishing
order of Steklov eigenfunctions on $\partial\Omega$.

\begin{theorem}
The vanishing order of Steklov eigenfunction on $\partial\Omega$ is
everywhere less than $C\lambda$ where $C$  is positive constant
depending only on $\Omega$. \label{th2}
\end{theorem}

The outline of the paper is as follows. Section 2 is devoted to
obtaining the appropriate Carleman estimates for a general equation
derived from the Steklov eigenvalue problem. In Section 3, a
three-ball theorem is presented. Then the doubling estimates on
solid ball follow.  In section 4, we obtain the doubling estimates
on the boundary and the vanishing order of Steklov eigenfunction.
The letters $C$, $\bar C$, $C_0$ and $C_1$ denote generic positive
constants and do not depends on $u$. They may also vary from line to
line.

\section{Carleman estimates }
In order to obtain the doubling estimates of Steklov eigenfunctions
on the boundary of $\Omega$, it is not easy to consider the
eigenfunctions solely on $\partial\Omega$.  we first do a reflection
across the boundary. The argument follows from \cite{BL}. Set
$$\Omega_\rho:=\{ y\in \mathbb R^n| dist\{y, \partial\Omega\} <\rho\}.       $$
Define
$$\Omega'_\rho:=\{x+t\nu(x)|x\in\partial\Omega, \ t\in (0, \
\rho)\}=\{y\in\mathbb R^n|dist(y, \partial\Omega)<\rho\}\backslash
\overline{\Omega},$$ where $\nu(x)$ is the unit outer normal at $x$.
Since $\partial \Omega$ is compact and smooth, there exists a map
$(x, t)\to x+t\nu(x)$ which is one-to-one from
$\partial\Omega\times(-\delta ,\delta)$ onto $\delta$--neighborhood
of $\partial\Omega$, where $\delta$ depends only on $\Omega$. Set
$$\tilde{\Omega}=\Omega_\delta\cap\partial\Omega\cap\Omega'_\delta $$
and
$$\hat{u}(x)=u(x)\exp\{\lambda d(x)\}$$
where $d(x)=dist(x, \partial\Omega)$ is the distance function.
Notice that $\hat{u}(x)=u(x)$ on $\partial\Omega$. From the equation
$(\ref{ste})$, one can check that $\hat{u}$ satisfies the equation
\begin{equation}
\left \{ \begin{array}{lll} div(A(x)\nabla\hat{u})+b(x)\cdot\nabla\hat{u}+q(x)\hat{u}=0  &\quad \mbox{in}\ \Omega_\delta,
\medskip\\
\frac{\partial \hat{u}}{\partial \nu}=0 &\quad \mbox{on}\
\partial\Omega_\delta,
\end{array}
\right.
\end{equation}
with
\begin{equation}
\left \{ \begin{array}{lll} A(x)=I,  \medskip\\
b(x)=-2\lambda \nabla d(x),   \medskip\\
q(x)=\lambda^2-\lambda\triangle d(x).
\end{array}
\right. \label{fffw}
\end{equation}
We consider the reflection from $\Omega_\delta$ to $\Omega'_\delta$.
Let $x=y+t\nu(y)$ with $y\in\partial\Omega$ and $t\in(-\delta, \
0)$. The reflection map is given by $$\Psi(x)=y-t\nu(y).$$ For
$x'\in\Omega'_\delta$, there exists $x$ such that $\Psi(x)=x'$.
Define
$$\hat{u}(x')=u(\Psi^{-1}(x'))=\hat{u}(x).$$
Then $\hat{u}$ satisfies
$$ div(A(x) \nabla \hat{u})+b(x)\cdot\nabla \hat{u}+q(x)\hat{u}=0 \quad \mbox{in} \ \Omega'_\delta$$
with $A(x)=(a_{ij})^n_{i,j=1},$
\begin{equation}
\left \{ \begin{array}{lll}
a_{ij}(x')=\sum^{n}_{k=1}\frac{\partial\Psi^i}{\partial x_k}(x)
\frac{\partial\Psi^j}{\partial x_k}(x), \nonumber \medskip \\
b^i(x')=-\sum^{n}_{j=1}\frac{\partial}{\partial x'_j} a^{ij}(x')+\triangle \Psi^i(x)+b(x)\cdot\nabla\Psi^i(x),  \nonumber \medskip \\
q(x')=q(x).
\end{array}
\right.
\end{equation}
Therefore, $\hat{u}$ satisfies the following equation
\begin{equation} div(A \nabla \hat{u})+b\cdot\nabla
\hat{u}+q\hat{u}=0 \quad \mbox{in} \ \tilde{\Omega} \label{starg}
\end{equation}
with \begin{equation}
\left \{ \begin{array}{lll} \|A\|_{L^\infty(\tilde{\Omega})}\leq C,  \medskip\\
\|b\|_{W^{1, \infty}(\tilde{\Omega})}\leq C\lambda, \medskip\\
\|q\|_{W^{1, \infty}(\tilde{\Omega})}\leq C\lambda^2, \label{core}
\end{array}
\right.
\end{equation}
where $C$ depends only on $\Omega$. One can also check that $A$ is
uniformly Lipshitz and the equation (\ref{starg}) is uniformly
elliptic in $\tilde{\Omega}$ with ellipticity constant depending
only on $\Omega$. Instead of studying $u$ in (\ref{ste}), we
investigate $\hat{u}$ in (\ref{starg}), then we reduce $\hat{u}$ to
the boundary of $\Omega$ by making use of the fact that
$\hat{u}(x)=u(x)$ on $\partial\Omega$.

 General speaking, Carleman estimates
and frequency functions are two major ways to obtain doubling
estimates and study strong unique continuation. The frequency
function is a local measure of ``degree" of a polynomial like
function and first observed by Almgren for harmonic functions.
Garofalo and Lin in \cite{GL} showed it powerful application in the
strong unique continuation problem. The doubling property
(\ref{dou2}) in \cite{BL} is achieved by sophisticated use of
frequency functions. Carleman estimates are weighted integral
inequalities which were introduced by Carleman. See e.g. \cite{JK},
\cite{KT}, \cite{LT}, \cite{S}, to just mention a few literature on
strong unique continuation using Carleman estimates. Our goal in
this section is to obtain the carleman estimates for (\ref{starg}).
Since $A(x)$ does not depend on $\lambda$, we can instead consider
the general elliptic equation
\begin{equation} a_{ij}(x)D_{ij}\hat{u}+b(x)\cdot\nabla \hat{u}+q(x)\hat{u}=0
\quad \mbox{in} \ \tilde{\Omega} \label{obje}
\end{equation}
with the coefficients satisfying the same conditions as that in
(\ref{core}). The summation over $i, j$ is understood. In the
following, we adapt the idea in deriving Carleman estimates in
\cite{LW} and \cite{BC} which extend Donnelly and Fefferman's
approach \cite{DF} for classical eigenfunctions.

To begin, we introduce polar coordinates in $\mathbb R^n\backslash
\{0\}$ by setting $x=r\omega$, with $r=|x|$ and
$\omega=(\omega_1,\cdots,\omega_n)\in S^{n-1}$. Moreover, we use a
new coordinate $t=\log r$. Then
$$ \frac{\partial }{\partial x_j}=e^{-t}(\omega_j\partial_t+  \Omega_j), \quad 1\leq j\leq n, $$
where $\Omega_j$ is a vector field in $S^{n-1}$. We also see that
the vector fields $\Omega_j$ satisfy
$$ \sum^{n}_{j=1}\omega_j\Omega_j=0 \quad \mbox{and} \quad
\sum^{n}_{j=1}\Omega_j\omega_j=n-1.$$ Since we are only interested
the local estimates of $\hat{u}$ and $r\to 0$ if and only if
$t\to-\infty$, we study the case that $t$ sufficiently closes to
$-\infty$. One can also see that
$$\frac{\partial^2}{\partial x_i\partial x_j}=e^{-2t}(\omega_i\partial_t-\omega_i+\Omega_i)(\omega_j\partial_t+\Omega_j),
\ \ 1\leq i\leq j\leq n.     $$ Note that
$$e^{2t}\triangle=\partial^2_t+(n-2)\partial_t+\triangle_\omega,         $$
where $\triangle_\omega=\sum_j\Omega^2_j$ is the Laplace-Beltrami
operator on $S^{n-1}$. We want to establish the Carleman estimates
in the neighborhood of $x\in\partial\Omega$. From (\ref{fffw}),  we
know
$$ a_{ij}(x)=\delta_{ij}+O(e^t), \ \ \mbox{as} \ t\to  -\infty.$$

Next we obtain the Carleman estimates with weight
$$\phi_\beta(x)=e^{-\beta\tilde{\phi}(x)}$$ with $\beta>0$ and
$$\tilde{\phi}(x)=\log|x|+\log((\log|x|)^2).$$ Since we have introduced
the new coordinate concerning $t$, by denoting $$\phi(t)=t+\log
t^2,$$ then $\tilde{\phi}(x)=\phi(\log|x|)$. We consider the local
estimates of any smooth supported function $u$ at
$x_0\in\partial\Omega$. Let $r(x)=|x-x_0|$.  Our Carleman estimate
concerning the general elliptic equations (\ref{starg}) is given as
follows.
\begin{proposition}
There exist $C_1$, $C_0$ and sufficiently small $r_0$ depending on
$\Omega$ such that for $u\in C^{\infty}_{0}(\mathbb
B_{r_0}(x_0)\backslash \{x_0\}),$ $x_0\in\partial\Omega$ and
$$\beta>C_1(1+\|b\|_{W^{1, \infty}}+{\|q\|}^{1/2}_{W^{1, \infty}}),$$
one has
\begin{eqnarray}
C_0\|\phi_\beta(a_{ij}D_{ij}u+b\cdot\nabla
u+qu)r^{\frac{4-n}{2}}\|^2 &\geq& \beta^3\|\phi_\beta {(\log
r)}^{-1}
u r^{-\frac{n}{2}}\|^2 \nonumber \medskip \\
&+&\beta\|\phi_\beta {(\log r)}^{-1}\nabla u r^{\frac{2-n}{2}}\|^2.
\label{carl} \end{eqnarray} \label{esti}
\end{proposition}

Based on the proof of proposition \ref{esti}, if $u$ has support
away from $x_0$, we will have the following corollary which is
useful in deriving the three-ball theorem and doubling estimates.

\begin{corollary} Besides the conditions in proposition \ref{esti},
assume further that
$$ supp \ u\subset \{x_0\in\partial\Omega, \  |x-x_0|\geq \tilde{\delta}>0 \},$$
we have
\begin{eqnarray}
C\|\phi_\beta(a_{ij}D_{ij}u+b\cdot\nabla
u+qu)r^{\frac{4-n}{2}}\|^2&\geq& \beta^3\|\phi_\beta {(\log r)}^{-1}
u r^{-\frac{n}{2}}\|^2+\beta\tilde{\delta}
\|\phi_\beta u r^{-\frac{1+n}{2}}\|^2 \nonumber\\
&+&\beta\|\phi_\beta {(\log r)}^{-1}\nabla u r^{\frac{2-n}{2}}\|^2.
\label{usf}
\end{eqnarray}
\label{cor2}
\end{corollary}

\begin{proof}[Proof of Proposition 1]

 Define the differential operators
$P=a_{ij}D_{ij}$ and $Q=b(x)\cdot\nabla$. Let
$$u=e^{\beta\phi(t)}v.$$ Set
$$ P_{\beta} v=e^{-\beta\phi(t)}P(e^{\beta\phi(t)}v), $$
$$Q_{\beta} v=e^{-\beta\phi(t)}Q(e^{\beta\phi(t)}v).$$
Then
\begin{equation}
\begin{array}{rll}
L_\beta v:=& e^{2t}P_{\beta}v+e^{2t}Q_{\beta}v+e^{2t}qv \nonumber \medskip\\
=&\partial_t^2 v+a v+\hat{a} \partial_t v+\triangle_\omega
v+e^{2t}qv+S(v)+R(v), \end{array}
\end{equation}
where
\begin{equation}
\left\{\begin{array}{rll}
a &=&\beta\phi^{''}+\beta^2(\phi')^2+(n-2)\beta \phi' \nonumber \medskip\\
&=&(1+2t^{-1})^2\beta^2+(n-2)\beta+2(n-2)t^{-1}\beta-2t^{-2}\beta, \nonumber \medskip\\
\hat{a} &=& 2\beta\phi'+(n-2) \nonumber \medskip\\
&=&2\beta+4\beta t^{-1}+(n-2),  \nonumber \medskip\\
S(v)&=&\sum_{j+|\alpha|\leq 2}C_{j\alpha}(t,
\omega)\partial_t^j\Omega^{\alpha}v+\bar
C(t,\omega)(\beta(\phi')^2+\beta\phi'')v+e^tb_i\omega_i\beta\phi'v, \nonumber \medskip\\
R(v)&=&\sum_{j+|\alpha|\leq 1}C_{j\alpha}(t, \omega)\beta
\phi'\partial_t^j\Omega^{\alpha}v+e^tb_i\omega_i\partial_t
v+e^tb_i\Omega_iv, \nonumber \medskip\\
C_{j\alpha}&=&O(e^t), \ \ \partial_tC_{j\alpha}=O(e^t), \ \
\Omega^\alpha
C_{j\alpha}=O(e^t), \nonumber \medskip\\
\bar C&=&O(e^t), \ \ \partial_t\bar C=O(e^t), \ \ \Omega^\alpha
\bar C=O(e^t). \nonumber \medskip\\

\end{array}
\right.
\end{equation}
In order to show (\ref{carl}) for $t$ close to $-\infty$, we only
need to prove the following
\begin{eqnarray}
C\int|e^{2t} P_\beta v+e^{2t}Q_\beta v+e^{2t}q v|^2 \,dtd\omega
&\geq& \beta\int t^{-2}|\partial_t v|^2
\,dtd\omega+\beta\sum_{j}\int
t^{-2}|\Omega_j v|^2\,dtd\omega  \nonumber \medskip\\
&+&\beta^3\int t^{-2} v^2 \,dtd\omega. \label{eqv}
\end{eqnarray}
By the triangle inequality,
\begin{equation}
\|L_\beta v\|^2\geq \frac{1}{2}\mathcal {A}-\mathcal{B}, \label{new}
\end{equation}
where \begin{eqnarray} \mathcal {A}=\|\partial_t^2 v
&+&\triangle_\omega v+ \sum_{j+|\alpha|= 2}C_{j\alpha}(t,
\omega)\partial_t^j\Omega^{\alpha}v +(2\beta\phi'+e^t
b_i\omega_i)\partial_t v+e^tb_i\Omega_i v \nonumber
\\ &+& (\beta^2(\phi')^2+(n-2)\beta\phi'+
e^t\beta\phi'b_i\omega_i+e^{2t}q) v \|^2 \nonumber
\end{eqnarray}
and
\begin{eqnarray}
\mathcal{B}=\|\beta\phi'' v&+&(n-2)\partial_t v+\sum_{j+|\alpha|\leq
1}C_{j\alpha}(t, \omega)\partial_t^j\Omega^{\alpha}v \nonumber
\\ &+&\sum_{j+|\alpha|\leq 1}C_{j\alpha}(t,
\omega)\beta\phi'\partial_t^j\Omega^{\alpha}v +\bar
C(t,\omega)(\beta(\phi')^2+\beta\phi'')v \|^2.
\end{eqnarray}

We first study $\mathcal {A}$. While $\mathcal {B}$ will be
incorporated into the leading term in $\mathcal {A}$. We decompose
$\mathcal {A}$ as
$$ \mathcal {A}=\mathcal {A}_1+\mathcal {A}_2+\mathcal {A}_3,$$
where
\begin{eqnarray}
\mathcal {A}_1=\|\partial_t^2 v &+&\triangle_\omega v+
\sum_{j+|\alpha|= 2}C_{j\alpha}(t,
\omega)\partial_t^j\Omega^{\alpha}v \nonumber
\\&+&(\beta^2(\phi')^2+(n-2)\beta\phi'+
e^t\beta\phi'b_i\omega_i+e^{2t}q) v \|^2, \nonumber
\end{eqnarray}
\begin{eqnarray} \mathcal {A}_2 =\| (2\beta\phi'+e^t b_i\omega_i)\partial_t
v+e^tb_i\Omega_i v \|^2, \nonumber \end{eqnarray}
\begin{eqnarray}
\mathcal {A}_3=2<\partial_t^2 v+\triangle_\omega v+
\sum_{j+|\alpha|= 2}C_{j\alpha}(t,
\omega)\partial_t^j\Omega^{\alpha}v
 (\beta^2(\phi')^2+(n-2)\beta\phi'+
e^t\beta\phi'b_i\omega_i+e^{2t}q) v \nonumber \\
(2\beta\phi'+e^t b_i\omega_i)\partial_t v+e^tb_i\Omega_i v>.
\nonumber
\end{eqnarray}
Let us calculate $\mathcal {A}_1$. Since $|\phi''|\leq 1$ as $t$ is
negatively large, \begin{equation} \mathcal {A}_1\geq
\frac{\tilde{\epsilon}}{\beta}\mathcal {A}_1'
\label{reg}\end{equation} with $0<\tilde{\epsilon}<1$ to be
determined, where
\begin{eqnarray} \mathcal
{A}_1'=\|\sqrt{|\phi''|}\Big( \partial_t^2 v &+&\triangle_\omega v+
\sum_{j+|\alpha|= 2}C_{j\alpha}(t,
\omega)\partial_t^j\Omega^{\alpha}v \nonumber
\\ &+&(\beta^2(\phi')^2+(n-2)\beta\phi'+
e^t\beta\phi'b_i\omega_i+e^{2t}q) v \Big)\|^2.
\end{eqnarray}
We split $\mathcal {A}_1'$ into three parts. $$ \mathcal
{A}_1'=\mathcal {J}_1+\mathcal {J}_2+\mathcal {J}_3,$$ where
\begin{eqnarray}
\mathcal {J}_1=\|\sqrt{|\phi''|}\big(\partial_t^2 v
+\triangle_\omega v+ \sum_{j+|\alpha|= 2}C_{j\alpha}(t,
\omega)\partial_t^j\Omega^{\alpha}v\big)\|^2, \nonumber
\end{eqnarray}
\begin{eqnarray}
\mathcal
{J}_2=\|\sqrt{|\phi''|}\big(\beta^2(\phi')^2+(n-2)\beta\phi'+
e^t\beta\phi'b_i\omega_i+e^{2t}q \big)v\|^2, \nonumber
\end{eqnarray}
\begin{eqnarray}
\mathcal {J}_3=2<\partial_t^2 v &+&\triangle_\omega v+
\sum_{j+|\alpha|= 2}C_{j\alpha}(t,
\omega)\partial_t^j\Omega^{\alpha}v \nonumber \\
&&|\phi''|\big(\beta^2(\phi')^2+(n-2)\beta\phi'+
e^t\beta\phi'b_i\omega_i+e^{2t}q \big)v>. \nonumber \end{eqnarray}
Note that $\mathcal {J}_1$ is nonnegative. Since $\phi'$ is close to
$1$ as $t\to -\infty$, by the triangle inequality,
\begin{eqnarray}
\mathcal {J}_2\geq
\beta^4\|\sqrt{|\phi''|}v\|^2-\|\sqrt{|\phi''|}\big(
(n-2)\beta\phi'+ e^t\beta\phi'b_i\omega_i+e^{2t}q\big)v\|^2.
\nonumber
\end{eqnarray}
Since we have chose that $$\beta>C_1(1+\|b\|_{W^{1,
\infty}}+{\|q\|}^{1/2}_{W^{1, \infty}}),$$ by Young's inequality,
\begin{eqnarray}
\|\sqrt{|\phi''|}\big( (n-2)\beta\phi'+
e^t\beta\phi'b_i\omega_i+e^{2t}q\big)v\|^2\leq \epsilon \beta^4
\|\sqrt{|\phi''|}v\|^2+C_\epsilon\|\sqrt{|\phi''|}v\|^2. \nonumber
\end{eqnarray}
Choosing $\epsilon$ appropriately small, we have
\begin{eqnarray}
\mathcal {J}_2&\geq& C\beta^4 \|\sqrt{|\phi''|}v\|^2-C\beta^2
\|\sqrt{|\phi''|}v\|^2  \ \nonumber \medskip\\
&\geq& C\beta^4 \|\sqrt{|\phi''|}v\|^2. \label{reg1}
\end{eqnarray}
We find a lower bound for $\mathcal {J}_3$. Using integration by
parts, it follows that
\begin{eqnarray}
\mathcal {J}_3&=&-2\int
|\phi''|\big(\beta^2(\phi')^2+(n-2)\beta\phi'+
e^t\beta\phi'b_i\omega_i+e^{2t}q \big) |\partial_t v|^2 \nonumber \\
&&-2\int
\partial_t\Big(|\phi''|\big(\beta^2(\phi')^2+(n-2)\beta\phi'+
e^t\beta\phi'b_i\omega_i+e^{2t}q \big)\Big)v\partial_t v \nonumber \\
&&-2 \sum_{j=1}^n\int |\phi''|\big(\beta^2(\phi')^2+(n-2)\beta\phi'+
e^t\beta\phi'b_i\omega_i+e^{2t}q\big)|\Omega_j v|^2 \nonumber \\
&&-2 \int |\phi''|\Omega_j\big(\beta^2(\phi')^2+(n-2)\beta\phi'+
e^t\beta\phi'b_i\omega_i+e^{2t}q\big)\Omega_j v v \nonumber \\
&&-2\int \sum_{j+|\alpha|= 1}C_{j\alpha}(t,
\omega)\partial_t^j\Omega^{\alpha}v \sum_{j+|\alpha|= 1}\Big(\big(
|\phi''|\beta^2(\phi')^2+(n-2)\beta\phi'+
e^t\beta\phi'b_i\omega_i+e^{2t}q \big)v\Big)\nonumber \\
&&\geq -C\beta^2\int |\phi''|(|\partial_t v|^2+\sum^n_{j=1}|\Omega_j
v|^2+|v|^2). \nonumber \end{eqnarray} Combining above estimates for
$\mathcal {J}_2$, $\mathcal {J}_3$ and (\ref{reg}) together, it
follows that
\begin{equation}
 \mathcal {A}_1\geq C\beta^3\tilde{\epsilon}\int |\phi''|v^2 -C\beta
\tilde{\epsilon}\int |\phi''|(|\partial_t v|^2+\sum^n_{j=1}|\Omega_j
v|^2+|v|^2). \label{small}
\end{equation}
Now we compute $\mathcal {A}_2$. Obviously, $$\mathcal {A}_2\geq
\frac{1}{\beta}\mathcal {A}_2.$$ By the triangle inequality,

$$\mathcal {A}_2\geq 4 \beta^2 \|\phi'\partial_t v||^2 -\|e^t b_i\omega_i\partial_t
v+e^tb_i\Omega_i v \|^2.$$ Application of Young's inequality yields

$$\|e^t b_i\omega_i\partial_t v+e^tb_i\Omega_i v \|^2\leq \epsilon \beta^2\|e^t \partial_t \|^2+
C_\epsilon \sum^n_{j=1}\beta^2
\|e^t\Omega_j v\|^2.$$ Note that $1+\frac{2}{T_0}\leq \phi'\leq 1$
as $-\infty\leq t<0$, which is close to $1$ as $T_0$ as negatively
large. It follows that \begin{equation}\mathcal {A}_2\geq
C\beta\|\partial_t v\|^2-C\beta\sum^n_{j=1}\|e^t\Omega_j v\|^2.
\label{regg}
\end{equation}
The calculation of $\mathcal {A}_3$ is lengthy. Using integration by
parts and the fact that
$$\lim_{t\to-\infty}-\frac{\phi''}{e^t}=\infty,$$
we have
\begin{eqnarray}
\mathcal {A}_3\geq C\beta\sum^n_{j=1}\int |\phi''||\Omega_j v|^2&-&
C\beta^3\int  |\phi''|e^tv^2-C\beta \int |\phi''||\partial_t v|^2
\nonumber \\ &-&C\beta^3\int e^t|v|^2. \label{smal}
\end{eqnarray}
Recall that $\mathcal {A}=\mathcal {A}_1+\mathcal {A}_2+\mathcal
{A}_3$. Together with the estimates (\ref{small}), (\ref{regg}) and
(\ref{smal}), by choosing $\tilde{\epsilon}$ appropriately small, we
can derive that
\begin{equation}
\mathcal {A}\geq C\beta\|\partial_t
v\|^2+C\beta^3\|\sqrt{|\phi''|}v|^2+C\beta\sum^n_{j=1}\|\sqrt{|\phi''|}|\Omega_jv\|^2.
\label{starr}
\end{equation}
Finally we estimate $\mathcal {B}$. By integration by parts,
\begin{eqnarray}
\mathcal{B}=\|\beta\phi'' v &+&(n-2)\partial_t
v+\sum_{j+|\alpha|\leq 1}C_{j\alpha}(t,
\omega)\partial_t^j\Omega^{\alpha}v \nonumber
\\ &+&\sum_{j+|\alpha|\leq 1}C_{j\alpha}(t,
\omega)\beta\phi'\partial_t^j\Omega^{\alpha}v +\bar
C(t,\omega)(\beta(\phi')^2+\beta\phi'')v \|^2 \nonumber \\
&&\leq C\big(\beta^2\|\sqrt{|\phi''|} v\|^2+\|\partial_t
v\|^2+\|e^t\partial_t v\|^2+ \sum^n_{j=1}\|e^t\Omega_j u\|^2\big).
\end{eqnarray}
We can check that each term in $\mathcal{B}$ can be incorporated in
the right hand side of (\ref{starr}). Therefore,
\begin{equation}
\|L_\beta v\|^2 \geq C\beta\|\sqrt{|\phi''|}\partial_t
v\|^2+C\beta^3\|\sqrt{|\phi''|}v\|^2+C\beta\sum^n_{j=1}\|\sqrt{|\phi''|}\Omega_jv\|^2.
\end{equation}
We are done with the proof of Proposition 1.
\end{proof}

The proof the corollary is inspired by the idea in \cite{BC}.

\begin{proof}[proof of Corollary \ref{cor2}]
The application of Cauchy-Schwartz inequality yields
\begin{eqnarray}
\int \partial_t (v^2)e^{-t}&=&2\int v\partial_t e^{-t} \nonumber \\
&\leq &2\big(\int (\partial_t v)^2
e^{-t}\big)^{\frac{1}{2}}\big(\int  v)^2 e^{-t}\big)^{\frac{1}{2}}.
\end{eqnarray}
On the other hand, by integration by parts,
\begin{equation}
\int \partial_t (v^2) e^{-t} =\int v^2 e^{-t}.
\end{equation}
Hence \begin{equation} \int v^2 e^{-t}\leq C\int (\partial_t v)^2
e^{-t}. \nonumber
\end{equation}
Since $u$ is supported in $\{x| \ \ |x-x_0|\geq \tilde{\delta}\}$,
\begin{equation} C\tilde{\delta} \int v^2 e^{-t}\leq \int (\partial_t
v)^2.
\end{equation}
Applying above estimates to (\ref{starr}), we know an another
refined lower bound for $\mathcal{A}$, that is,
$$ C\mathcal{A}\geq\tilde{\delta}\beta\int v^2 e^{-t}.$$
Carrying out the similar proof as Proposition 1, we know that
\begin{eqnarray}
\|L_\beta v\|^2 \geq C\beta \|\sqrt{|\phi''|}\partial_t
v\|^2&+&C\beta^3\|\sqrt{|\phi''|}v\|^2+C\beta\sum^n_{j=1}\|\sqrt{|\phi''|}\Omega_jv\|^2
\nonumber \\
&+&\tilde{\delta}\beta\| v e^{-t/2}\|^2,
\end{eqnarray}
which implies the corollary.
\end{proof}

\section{ Doubling estimates on solid balls}
Following from the similar argument in \cite{BC}, we have the
following three-ball theorem. Its argument is also similar to the
proof of Proposition 2 below.

\begin{lemma}
There exist constants $r_0,$ $C$ and $0<\gamma<1$ depending only on
$\Omega$ such that for any solutions of (\ref{obje}),
$0<R<r_0<\delta$, and $x_0\in\partial\Omega$, one has
\begin{equation}
\int_{\mathbb B_R(x_0)}{\hat u}^2\leq e^{C(1+\|b\|_{W^{1,
\infty}}+{\|q\|}^{1/2}_{W^{1, \infty}})}(\int_{\mathbb
B_{2R}(x_0)}{\hat u}^2)^{1-\gamma}(\int_{\mathbb B_{R/2}(x_0)}{\hat
u}^2)^{\gamma}.
\end{equation}
\label{three}
\end{lemma}
In the following, we want to derive estimates that controls $L^2$
norm of $\hat{u}$ in $\tilde{\Omega}$ in terms of $L^2$ norm of
$\hat{u}$ in a small ball centered at $\partial\Omega$. In the
presentation, sometimes we drop the center of $\mathbb B_r$ if there
is no misunderstanding. By the lemma 3.7 in \cite{BL}, for any
$0<r<\delta$, there exists a point $x_\star \in\partial\Omega$ such
that
\begin{equation} \int_\Omega u^2\leq C r^{-2n+1}\int_{\mathbb B_r(
x_\star)\cap\Omega} u^2 \label{foll}
\end{equation}
for some $C$ depending only on $\Omega$. Recall that
$\hat{u}=e^{\lambda d(x)} u$. Obviously
\begin{equation}
\int_{\mathbb B_r( x_\star)\cap\Omega} u^2\leq \int_{\mathbb B_r(
x_\star)\cap\tilde{\Omega}} \hat{u}^2. \label{foln}
\end{equation}
Thus,
\begin{equation}
\int_\Omega u^2\leq C r^{-2n+1}\int_{\mathbb B_r(
x_\star)\cap\tilde{\Omega}} \hat{u}^2. \label{sss}
\end{equation}
 We can also check that
\begin{equation}
\int_{\Omega_\delta}
\hat{u}^2(x)=\int_{\Omega'_\delta}\hat{u}^2(x')|\nabla \Psi(x')|,
\end{equation}
where $|\nabla \Psi(x')|>C$ in $\Omega'_\delta$ with $C$ depending
only on $\Omega$, that is,
\begin{equation}
\int_{\Omega_\delta} \hat{u}^2(x)\geq
C\int_{\Omega'_\delta}\hat{u}^2(x'). \label{folm}
\end{equation}
Since
\begin{equation}
\int_\Omega u^2\geq
e^{-2\lambda\delta}\int_{\Omega_\delta}\hat{u}^2,
\end{equation}
it follows from (\ref{sss}) and (\ref{folm}) that
\begin{equation}
\int_{\tilde{\Omega}} \hat{u}^2\leq C e^{2\lambda\delta}
r^{-2n+1}\int_{\mathbb B_r( x_\star)\cap\tilde{\Omega}}\hat{u}^2.
\label{imp}
\end{equation}

\begin{lemma}
For any $0<r<r_0<\delta$, there exists $C_r$ such that for any
$x_0\in\partial\Omega$, the following holds
\begin{equation}
\int_{\mathbb B_r(x_0)}\hat{u}^2\geq
e^{-C_r\lambda}\int_{\tilde{\Omega}}\hat{u}^2.
\end{equation}
\label{lem2}
\end{lemma}
\begin{proof}
Without loss of generality, we may denote
$\int_{\tilde{\Omega}}\hat{u}^2=1$. Let $\bar x\in\partial\Omega$ be
the point where $\|\hat{u}\|_{\mathbb B_r(\bar
x)}=\sup_{x\in\partial\Omega}\|\hat{u}\|_{\mathbb B_r(x)}.$ From
(\ref{imp}),
\begin{equation}
\|\hat{u}\|_{\mathbb B_r(\bar x)}\geq e^{-2\lambda\delta} C_{r}
\label{impp}
\end{equation}
with $C_r$ depending on $\Omega$ and $r$. Applying the three-ball
theorem (i.e. Lemma \ref{three}) at any $x\in\partial\Omega$, it
follows that
\begin{eqnarray}
\|\hat{u}\|_{\mathbb B_{r/2}(x)}&\geq& e^{-C(1+\|b\|_{W^{1,
\infty}}+{\|q\|}^{1/2}_{W^{1,
\infty}})}\|\hat{u}\|^{\frac{1}{\gamma}}_{\mathbb B_r(x)} \nonumber
\\ &\geq & e^{-C\lambda}\|\hat{u}\|^{\frac{1}{\gamma}}_{\mathbb
B_r(x)}, \label{kkk}
\end{eqnarray}
where we have the assumptions for $b$ and $q$ in (\ref{core}).

 Define
a sequence of point $x_0$, $x_1$, $\cdots$, $x_m=\bar x$ on
$\partial\Omega$ such that
$$\mathbb B_{r/2}(x_{i+1})\subset  \mathbb B_{r}(x_{i}) \quad
\mbox{for} \ i=0, \cdots, m-1.$$ We can see that the number $m$
depends on $\Omega$ and $r$. From (\ref{kkk}), we have
$$\|\hat{u}\|_{\mathbb B_{r/2}(x_i)}\geq e^{-C\lambda}\|\hat{u}\|^{\frac{1}{\gamma}}_{\mathbb
B_{r/2}(x_{i+1})}$$ for $i=0, 1, \cdots, m-1.$ By iteration and
(\ref{impp}), it follows that
\begin{equation}
\int_{\mathbb B_r(x_0)}\hat{u}^2\geq
e^{-C_r\lambda}\int_{\tilde{\Omega}}\hat{u}^2. \label{kao}
\end{equation}
We are done with Lemma \ref{lem2}.
\end{proof}

Let
$$ A_{r, \, 2r}=\{x\in\tilde{\Omega}| r\leq |x-x_0|\leq 2r, \
\mbox{with} \ x_0\in\partial\Omega\}. $$ We can find that there
exits some $x_i\in A_{r, \, 2r}\cap\partial\Omega$ such that
$\mathbb B_{{r}/{3}}(x_i)\subset A_{r,\, 2r}$. It yields that
\begin{equation}
\int_{A_{r, 2r}}\hat{u}^2\geq
e^{-C_r\lambda}\int_{\tilde{\Omega}}\hat{u}^2. \label{rkao}
\end{equation}

We are ready to derive the doubling estimates for $\hat{u}$ on solid
balls on $\tilde{\Omega}$.
\begin{proposition}
There exist positive constants $r_0$, $C$ depending only on $\Omega$
such that for any solution $\hat{u}$ in (\ref{starg}),
\begin{equation}
\int_{\mathbb B_{2r}(x_0)} |\hat{u}|^2\leq e^{C\lambda}
\int_{\mathbb B_{r}(x_0)} |\hat{u}|^2
\end{equation}
for any $r<r_0$ and $x_0\in\partial\Omega$. \label{pro2}
\end{proposition}
\begin{proof}
 Let
$R=\frac{r_0}{6}$, where $r_0$ is the one in the three-ball theorem.
Let $ 0<4\rho<\frac{r_0}{18}$. Select a small cut-off function with
$0\leq \psi\leq 1$ satisfying the following properties
\begin{equation}
\left \{ \begin{array}{lll} \psi(x)=0 \quad \mbox{if} \quad
r(x)<\rho \ \mbox{and} \
r(x)>R, \nonumber \medskip \\
\psi(x)=1 \quad \mbox{if} \quad \frac{3\rho}{2}<r(x)<\frac{R}{3},
\nonumber \medskip\\
|\nabla \psi(x)|\leq \frac{C}{\rho}, \ \mbox{and} \  |\nabla^2
\psi(x)|\leq \frac{C}{\rho^2} \ \mbox{if} \
\rho<r(x)<\frac{3\rho}{2},
\nonumber \medskip\\
|\nabla \psi(x)|\leq \frac{C}{R}, \ \mbox{and} \  |\nabla^2
\psi(x)|\leq \frac{C}{R^2} \ \mbox{if} \ \frac{R}{3}<r(x)<R.
\end{array}
\right.
\end{equation}
Now we apply the Carleman estimates in (\ref{usf}) with $\psi
\hat{u}$. Dropping the last term in the right hand side of
(\ref{usf}) yields
\begin{eqnarray}
\beta\rho\int e^{-2\beta\tilde{\phi}}|\psi \hat{u}|^2r^{-n-1} &+&
\beta^3\int (\log  r)^{-2}e^{-2\beta\tilde{\phi}}|\psi
\hat{u}|^2r^{-n}  \nonumber \\ &\leq C& \int
e^{-2\beta\tilde{\phi}}r^{4-n}\big(a_{ij}D_{ij}(\psi
\hat{u})+b_iD_i(\psi
\hat{u})+q\hat{u}\big)^2 \nonumber \\
 &\leq C& \int e^{-2\beta\tilde{\phi}}r^{4-n}\big(a_{ij}D_{ij}\psi \hat{u}+2a_{ij}D_i\hat{u}D_j\psi
 +b_iD_i\psi \hat{u}
 \big)^2.
\end{eqnarray}
By the properties of $\psi$ and the fact that $R$ depending only on
$\Omega$ and $\beta>1$, we have
\begin{eqnarray}
\int_{A_{\frac{3\rho}{2}, 4\rho}\cup A_{\frac{R}{12}, \frac{R}{6}}}
e^{-2\beta\tilde{\phi}}\hat{u}^2r^{-n} &\leq&
C(1+\lambda)^2\int_{A_{\rho, \frac{3\rho}{2}}\cup A_{\frac{R}{3},
R}}
e^{-2\beta\tilde{\phi}}\hat{u}^2r^{-n} \nonumber \\
&+ & C(\rho^2\int_{A_{\rho}, {\frac{3\rho}{2}}}
e^{-2\beta\tilde{\phi}}|\nabla
\hat{u}|^2r^{-n}+R^2\int_{A_{\frac{R}{3}, R}}
e^{-2\beta\tilde{\phi}}|\nabla \hat{u}|^2r^{-n}), \label{boo}
\end{eqnarray}
where the assumption in (\ref{core}) has been used. Recall the
following elliptic estimates for $\hat{u}$ in (\ref{starg}),
\begin{equation}
\|\nabla u\|_{\mathbb B_{ar}}\leq
C(\frac{1}{(1-a)r}+\|b\|_{L^\infty}+\|q\|_{L^\infty}^{1/2})\|
u\|_{\mathbb B_{r}}, \quad \mbox{for} \ 0<a<1.
\end{equation}
Applying elliptic estimates to the last two terms in the right hand
side of (\ref{boo}) and the decreasing property of $-\tilde{\phi}$,
we obtain
\begin{eqnarray}
e^{-2\beta\tilde{\phi}(4\rho)}\rho^{-n} \int_{A_{\frac{3\rho}{2},
4\rho}}
|\hat{u}|^2+e^{-2\beta\tilde{\phi}(\frac{R}{6})}R^{-n}\int_{A_{\frac{R}{12},
\frac{R}{6}}} |\hat{u}|^2 &\leq & C(1+\lambda)^4\big(
e^{-2\beta\tilde{\phi}(\rho)}\rho^{-n}\int_{\mathbb
B_{2\rho}}|\hat{u}|^2
\nonumber \\
&+&e^{-2\beta\tilde{\phi}(\frac{R}{3})}R^{-n}\int_{\mathbb
B_{\frac{3R}{2}}}|\hat{u}|^2 \big). \label{abo}
\end{eqnarray}
Adding $e^{-2\beta\tilde{\phi}(4\rho)}\rho^{-n} \int_{\mathbb
B_{\frac{3\rho}{2}}} |\hat{u}|^2 $ to both sides of (\ref{abo}), we
have
\begin{eqnarray}
e^{-2\beta\tilde{\phi}(4\rho)}\rho^{-n}\int_{\mathbb B_{4\rho}}
|\hat{u}|^2 +
e^{-2\beta\tilde{\phi}(\frac{R}{6})}R^{-n}\int_{A_{\frac{R}{12},
\frac{R}{6}}} |\hat{u}|^2 &\leq & C(1+\lambda)^4\big(
e^{-2\beta\tilde{\phi}(\rho)}\rho^{-n}\int_{\mathbb
B_{2\rho}}|\hat{u}|^2
\nonumber \\
&+&e^{-2\beta\tilde{\phi}(\frac{R}{3})}R^{-n}\int_{\mathbb
B_{\frac{3R}{2}}}|\hat{u}|^2 \big). \label{abso}
\end{eqnarray}
We want to absorb the last term in the right hand side of
(\ref{abso}) into the left hand side. Let $\beta$ satisfy that
\begin{equation}
C(1+\lambda)^4e^{-2\beta\tilde{\phi}(\frac{R}{3})}\int_{\mathbb
B_{\frac{3R}{2}}}|\hat{u}|^2\leq
\frac{1}{2}e^{-2\beta\tilde{\phi}(\frac{R}{6})}\int_{A_{\frac{R}{12},
\frac{R}{6}}} |\hat{u}|^2. \nonumber
\end{equation}
In order to make it true, we only need to select $\beta$ to be
\begin{equation}
\beta=\frac{1}{2(\tilde{\phi}(\frac{R}{6})-\tilde{\phi}(\frac{R}{3}))}\ln
\frac{\int_{A_{\frac{R}{12}, \frac{R}{6}}} |\hat{u}|^2}{
2C(1+\lambda)^4\int_{\mathbb
B_{\frac{3R}{2}}}|\hat{u}|^2}+C_1(1+\|b\|_{W^{1,
\infty}}+{\|q\|}^{1/2}_{W^{1, \infty}}). \label{who}\end{equation}
Thus,
\begin{eqnarray}
e^{-2\beta\tilde{\phi}(4\rho)}\rho^{-n}\int_{\mathbb B_{4\rho}}
|\hat{u}|^2 &+&\frac{1}{2}
e^{-2\beta\tilde{\phi}(\frac{R}{6})}R^{-n}\int_{A_{\frac{R}{12},
\frac{R}{6}}} |\hat{u}|^2 \nonumber \\ &\leq & C(1+\lambda)^4
e^{-2\beta\tilde{\phi}(\rho)}\rho^{-n}\int_{\mathbb
B_{2\rho}}|\hat{u}|^2. \label{absoo}
\end{eqnarray}
Dropping the second term in the left hand side leads to
\begin{equation}
e^{-2\beta\tilde{\phi}(4\rho)}\int_{\mathbb B_{4\rho}}
|\hat{u}|^2\leq C(1+\lambda)^4
e^{-2\beta\tilde{\phi}(\rho)}\int_{\mathbb B_{2\rho}}
|\hat{u}|^2.\label{ddd}
\end{equation}
Let
$D_R=(\tilde{\phi}(\frac{R}{6})-\tilde{\phi}(\frac{R}{3}))^{-1}$. We
know that $0<A^{-1}<-D_R<A$ where $A$ does not depend on $R$. Note
that $0<\tilde{\phi}(4\rho)-\tilde{\phi}(\rho)<C$. It follows from
(\ref{who}) and (\ref{ddd}) that
\begin{equation}
\int_{\mathbb B_{4\rho}} |\hat{u}|^2\leq e^{C\lambda} \big(
\frac{\int_{\mathbb
B_{\frac{3R}{2}}}|\hat{u}|^2}{\int_{A_{\frac{R}{12}, \frac{R}{6}}}
|\hat{u}|^2} \big)^{A^{-1}} \int_{\mathbb B_{2\rho}} |\hat{u}|^2.
\end{equation}
From (\ref{rkao}), we arrive at
\begin{equation}
\int_{\mathbb B_{4\rho}} |\hat{u}|^2\leq e^{C\lambda} \int_{\mathbb
B_{2\rho}} |\hat{u}|^2.
\end{equation}
This completes the proof of proposition \ref{pro2}.
\end{proof}

\section{Doubling estimates and vanishing order on the boundary}
In the above section, the doubling property for $\hat{u}$ on solid
balls is given. Next we obtain a doubling inequality on the boundary
of $\Omega$. We do a scaling for $\hat{u}$ in $\mathbb
B_{\frac{r_0}{\lambda}}(x_0)$. Let
$$\bar u(x)= \hat{u}(x_0+\frac{x}{\lambda})$$ for $x\in \mathbb
B_{r_0}$. From (\ref{starg}), $\bar u$ satisfies
\begin{equation}
div (\bar A\nabla \bar u)+\bar b\cdot \nabla \bar u+\bar q \bar u=0
\quad a.e. \ \ \mbox{in} \ \mathbb B_{r_0} \label{resc}
\end{equation}
where
 \begin{equation}
\left \{ \begin{array}{lll} \bar A(x)=A(x_0+\frac{x}{\lambda}), \medskip\\
\bar b(x)=\lambda^{-1} b(x_0+\frac{x}{\lambda}), \medskip\\
\bar q(x)=\lambda^{-2} q(x_0+\frac{x}{\lambda}).
\end{array}
\right.
\end{equation}
We can also check that \\
(i) there exists some positive constant $\varrho$ such that
$$ \varrho |\xi|^2\leq \sum {\bar a_{ij}(x)}\xi_i\xi_j $$
where $\bar A(x)=\big(\bar a_{ij}(x)\big)^n_{i,j=1}$. \\
(ii) There exists constant $K_1$ such that
$$ |\bar a_{ij}(x)-\bar a_{ij}(y)|\leq K_1|x-y|. $$\\
(iii) There exists constant $K_2$ such that
$$ \sum_{i,j}\|\bar a_{i,j}\|_{L^\infty(\mathbb
B_1)}+\sum_{i}\|\bar b_i\|_{L^\infty(\mathbb
B_1)}+\|q\|_{L^\infty(\mathbb B_1)}\leq K_2.$$

Based on the doubling estimates on $\tilde{\Omega}$, we can prove
the doubling estimates on $\partial\Omega$. We follow the idea in
\cite{BL}.  The following lemma obtained in \cite{Lin} is the
connection between the $L^2$ norm of solutions on the boundary and
solid balls. For completeness of the presentation, we state the
lemma.
\begin{lemma}
Assume that $u$ is the solution of (\ref{resc}) in $\mathbb B_1^+$
with coefficient satisfying (i)-(iii) and $ \| u\|_{L^2(\mathbb
B_1^+)}\leq 1$. Suppose that
\begin{equation}
\|u\|_{H^1(\Gamma)}+\|\frac{\partial u}{\partial
x_n}\|_{L^2(\Gamma)}\leq \epsilon<<1,
\end{equation}
where $\Gamma=\{(x', 0)\in\mathbb R^n, |x'|\leq {2}/{3}\}$, then
$$\|u\|_{L^2(\mathbb B^+_{{1}/{2}})}\leq C\epsilon^{\alpha}     $$
for the constants $C$, $\alpha$ depending on $\varrho$, $K_1$, $K_2$
and $n$. \label{lin}
\end{lemma}
The following J.L. Lions-type lemma is shown in \cite{BL}.
\begin{lemma}
Let $u\in H^2(\mathbb R^n)$ and the trace of $u$ on $\{x\in \mathbb
R^n|x_n=0\}=\mathbb R^{n-1}$ be denoted by $u$. Then there exist a
constant $C$ depending on $n$ such that for any $\epsilon>0$,
$$\|\nabla u\|_{L^2(\mathbb R^{n-1})}\leq \epsilon \|u\|_{H^2(\mathbb R^n)}+\frac{C}{\epsilon^2}\|u\|_{L^2(\mathbb R^{
n-1})}.$$ \label{lions}
\end{lemma}
Now we are ready to show the proof of Theorem \ref{th1}.
\begin{proof}[Proof of Theorem \ref{th1}] Since $u=\hat{u}$ on $\partial\Omega$, in order to prove that
$$ \int_{\mathbb B_{2r}(x)\cap\partial\Omega}u^2\leq e^{C\lambda}\int_{\mathbb
B_{r}(x)\cap\partial\Omega}u^2 $$ for $r\leq \frac{r_0}{\lambda}$
and $x\in\partial\Omega$, by letting $x=x_0\in\partial\Omega$, it
suffices to prove \begin{equation}\int_{\mathbb
B_{2r}\cap\partial\hat{\Omega}}\bar u^2\leq e^{C\lambda}
\int_{\mathbb B_r\cap\partial\hat{\Omega}} \bar u^2 \label{goal}
\end{equation} for $r\leq r_0$ with $\hat{\Omega}=\{ x|
x_0+\frac{x}{\lambda}\in\Omega\}$ and $\mathbb B_{r}$ centered at
origin. From Proposition 2, there exist $r_0$, $C$ depending on
$\Omega$ and $n$ such that
$$\int_{\mathbb B_{2r}(x_0)} \hat{u}^2\leq e^{C\lambda}\int_{\mathbb B_{r}(x_0)} \hat{u}^2 $$
for $r\leq \frac{r_0}{2}$. It is true that
\begin{equation}\int_{\mathbb B_{2r}} \bar{u}^2\leq
e^{C\lambda}\int_{\mathbb B_{r}} \bar{u}^2. \label{soli}
\end{equation}

Next we want to flatten the hypersurface $\partial\Omega$. Since
$\partial\Omega$ is smooth, we can regard $\mathbb
B_{\frac{r_0}{\lambda}(x_0)}\cap\partial\Omega$ as a smooth function
with $C^2$ norm bounded by a constant $C_0$ independent of $x_0$ and
$\lambda$. After rescaling, we assume that
$$\mathbb B_{r_0}\cap \partial\Omega_{x_0, \lambda}=\{ (x',\, x_n)\in\mathbb B_{r_0}
| x'\in\mathbb B^{n-1}_{r_0}, \ x_n=\Phi(x')\}, $$ where $\Phi\in
C^2(\mathbb B^{n-1}_{r_0}),$ $\Phi(0)=0$, $\nabla\Phi(0)=0$ and
$\|\Phi\|_{C^2}\leq C_0$. Furthermore, $\|\Phi\|_{C^1(\mathbb
B^{n-1}_{r_0})}\leq \epsilon$ with $\epsilon$ small. Introducing the
injective map $F: \mathbb B^{n-1}_{r_0}\times \mathbb R\to \mathbb
R^n$ as
$$ F(x', \,x_n)=(x', \, x_n+\Phi(x')).$$
We can check that
$$\mathbb B_{\frac{r}{1+\epsilon}}\subset F(\mathbb B_r)\subset
\mathbb B_{(1+\epsilon)r},$$
$$F(\mathbb B^{n-1}_{\frac{r}{1+\epsilon}}\times \{0\})\subset \mathbb B_r\cap\partial\Omega_{x_0,\,\lambda}\subset
F(\mathbb B^{n-1}_{(1+\epsilon)r}\times \{0\}).$$ If we show that
\begin{equation}
\int_{F(\mathbb B^{n-1}_{2r}\times\{0\})} \bar u^2\leq e^{C\lambda}
\int_{F(\mathbb B^{n-1}_{r}\times\{0\})} \bar u^2, \label{purp}
\end{equation}
then it implies (\ref{goal}). We define
$$w(x)=\bar u(F(x)) \quad \mbox{for} \ x\in\mathbb B_{\frac{r_0}{1+\epsilon}}.      $$
To prove (\ref{purp}), by the area formula, it is enough to show
that
\begin{equation}
\int_{\mathbb B^{n-1}_{2r}\times \{0\}} w^2\leq e^{C\lambda}
\int_{\mathbb B^{n-1}_{r}\times \{0\}} w^2
\label{doi}\end{equation}for $0<r<\frac{r_0}{2(1+\epsilon)}$, where
$r_0$ depends only on $\Omega$. By the same argument, it follows
from (\ref{soli}) that
\begin{equation}
\int_{\mathbb B_{2r}} w^2\leq e^{C\lambda} \int_{\mathbb B_r} w^2
\label{soli2}
\end{equation}
for $0<r<\frac{r_0}{2(1+\epsilon)}$. We can check that $w$ solve
\begin{equation}
div(\tilde{A}\nabla w)+\tilde{b}\nabla w+\tilde{q} w=0 \label{eqn}
\end{equation} where $\tilde{A}$, $\tilde{b}$ and $\tilde{q}$
satisfy (i)-(iii) with similar bounds independent of $\lambda$. We
fixed $r<\frac{r_0}{2(1+\epsilon)}$ so that the condition
(\ref{soli2}) holds. We do a rescaling to consider fixed radius. Let
\begin{equation} \bar{w}(x)=\tilde{C}w(rx)
\label{back}
\end{equation}
where the constant $\tilde{C}$ is chosen so that
\begin{equation}\int_{\mathbb B_2}\bar{w}(x)=1.
\label{back1}\end{equation}
 By the doubling estimates
in (\ref{soli2}), it follows that
\begin{equation}
\int_{\mathbb B_{{1}/{2}}}\bar{w}\geq e^{-C\lambda}. \label{ston}
\end{equation}
We can see that $\bar{w}$ also satisfies (\ref{eqn}). If
$$\|\bar{w}\|_{H^1(\mathbb B^{n-1}_{2/3})}+\|\frac{\partial \bar{w}}{\partial
x_n}\|_{L^2(\mathbb B^{n-1}_{2/3})}=\tilde{\epsilon}<<1,
$$  by Lemma \ref{lin} and (\ref{back1}), we have
$$ \|\bar{w}\|_{L^2(\mathbb B_{{1}/{2}})}\leq C\tilde{\epsilon}^\alpha.         $$
Due to the lower bound (\ref{ston}), we obtain
$$\tilde{\epsilon}>e^{\frac{-C\lambda}{\alpha}}C^{\frac{-1}{\alpha}},       $$
that is,
$$\|\bar{w}\|_{H^1(\mathbb B^{n-1}_{2/3})}+\|\frac{\partial \bar{w}}{\partial
x_n}\|_{L^2(\mathbb B^{n-1}_{2/3})}>e^{-C\lambda}.$$ Since the
normal derivative of $\bar v$ vanishes on
$\partial\Omega_{x_0,\lambda}$, using area formula,
$$ \|\nabla \bar{w}\|_ {L^2(\mathbb B^{n-1}_{2/3})}\geq \frac{1}{C} \|\frac{\partial \bar{w}}{\partial
x_n}\|_{L^2(\mathbb B^{n-1}_{2/3})}.$$ Therefore,
\begin{equation}
\|\bar{w}\|_{H^1(\mathbb B^{n-1}_{2/3})}\geq e^{-C\lambda}.
\label{fain}
\end{equation}
We claim that the estimate $$\|\bar{w}\|_{H^1(\mathbb
B^{n-1}_{2/3})}\geq \bar\epsilon$$ implies that
$$\|\bar{w}\|_{L^2(\mathbb B^{n-1}_{5/6})}\geq
\bar \epsilon^3/C.$$ Introduce a cut-off function $\psi\in
C^{\infty}_0(\mathbb R^{n})$ such that $\psi=1$ on $\mathbb B_{2/3}$
and $\psi=0$ on $\mathbb R^n\backslash {\mathbb B_{5/6}}$. Consider
$\bar{w}'=\bar{w}\cdot\psi$, we obtain for $\eta>0$,
\begin{eqnarray}
\|\nabla \bar{w}\|_{L^2(\mathbb B^{n-1}_{2/3})}&\leq & \|\nabla
\bar{w}'\|_{L^2(\mathbb R^{n-1})} \nonumber  \\
&\leq &\eta \|\bar{w}'\|_{H^2(\mathbb
R^n)}+\frac{C}{\eta^2}\|\bar{w}'\|_{L^2(\mathbb R^{ n-1})}
\nonumber \\
& \leq &C \eta+ \frac{C}{\eta^2}\|\bar{w}\|_{L^2(\mathbb
B^{n-1}_{5/6})} \nonumber
\end{eqnarray}
where we have used the interior elliptic estimates and Lemma
\ref{lions}. Let $\eta=\frac{\bar\epsilon}{2C}$, we have
\begin{eqnarray}
\bar\epsilon-\|\bar{w}\|_{L^2(\mathbb B^{n-1}_{2/3})}& \leq &
\|\nabla \bar{w}\|_{L^2(\mathbb B^{n-1}_{2/3})} \nonumber \\
&\leq &\frac{\bar\epsilon}{2}+ \frac{C}{\bar\epsilon^2}
\|\bar{w}\|_{L^2(\mathbb B^{n-1}_{5/6})}. \nonumber \end{eqnarray}
Then the claim follows. So (\ref{fain}) implies that
\begin{equation}\|\bar{w}\|_{L^2(\mathbb B^{n-1}_{5/6})}\geq e^{-C\lambda}.
\label{ggg}
\end{equation} Thanks to the trace theorem and interior elliptic estimates,
\begin{eqnarray}\|\bar w\|_{L^2(\mathbb B^{n-1}_{5/3})}&\leq& C\|\bar
w\|_{H^1(\mathbb B_{5/3})} \nonumber \\
&\leq &C\|\bar w\|_{L^2(\mathbb B_{2})}\leq C. \nonumber
\end{eqnarray}
Then
$$\|\bar w\|_{L^2(\mathbb B^{n-1}_{5/6})}\geq e^{-C\lambda}\|\bar w\|_{L^2(\mathbb B^{n-1}_{5/3})},        $$
that is,
$$\| w\|_{L^2(\mathbb B^{n-1}_{5r/6})}\geq e^{-C\lambda}\|  w\|_{L^2(\mathbb B^{n-1}_{5r/3})}$$
for $r\leq \frac{r_0}{2(1+\epsilon)}$. Now we fix $\epsilon$. For
instance, let $0<\epsilon<\frac{1}{10}$. From (\ref{doi}), we arrive
at Theorem \ref{th1}.
\end{proof}

Thanks to the argument in proving Theorem 1, we present the proof of
corollary \ref{cor1}.
\begin{proof}[Proof of Corollary \ref{cor1}]
We consider any point $x_0\in \partial\Omega$. For convenience, we
may assume $x_0$ to be origin. On one hand, by (\ref{back}) and
(\ref{back1}),
\begin{eqnarray}
\frac{1}{\tilde{C}}=\int_{\mathbb B_2}
w^2(rx)&=&\frac{1}{r^n}\int_{\mathbb B_{2r}} \bar u(F(x)) \nonumber \\
&\geq &\frac{\lambda^n}{r^n}\int_{\mathbb
B_\frac{2r}{(1+\epsilon)\lambda}}\hat{u}^2.
\end{eqnarray}
On the other hand, it follows from (\ref{ggg}) that \begin{eqnarray}
\frac{e^{-C\lambda}}{\tilde{C}}\leq
{\frac{1}{\tilde{C}}}\int_{\mathbb B^{n-1}_{5/6}}\bar
w&=&\frac{1}{r^{n-1}}\int_{\mathbb B^{n-1}_{{5r}/6}}\bar u(F(x))
\nonumber \\
&\leq &\frac{\lambda^{n-1}}{r^{n-1}}\int_{\mathbb
B^{}_{{5(1+\epsilon)r}/6\lambda}\cap\partial\Omega}\hat{u}^2
\nonumber \\ &=& \frac{\lambda^{n-1}}{r^{n-1}}\int_{\mathbb
B^{}_{{5(1+\epsilon)r}/6\lambda}\cap\partial\Omega}{u}^2.
\end{eqnarray}
By taking $\epsilon$ appropriately small, we have
\begin{equation}
e^{C\lambda}\frac{r}{\lambda}\int_{\mathbb
B_\frac{r}{\lambda}\cap\partial\Omega} u^2\geq \int_{\mathbb
B_{\frac{5r}{3\lambda}}}\hat{u}^2. \label{kin}
\end{equation}
Applying the elliptic estimates \cite{GT} (Theorem 8.17) for
(\ref{starg}),
\begin{equation}
e^{C\sqrt{\lambda}}r^{\frac{-n}{2}}\|\hat{u}\|_{L^2(\mathbb
B_{\frac{5r}{3}})}\geq \|\hat{u}\|_{L^\infty(\mathbb
B_{\frac{4r}{3}})}. \label{eee}
\end{equation}
Since
\begin{equation}
\|\hat{u}\|_{L^\infty(\mathbb B_{\frac{4r}{3}})}\geq
\|u\|_{L^\infty(\mathbb B_{\frac{4r}{3}}\cap\partial\Omega)}.
\label{fff}
\end{equation}
We conclude from (\ref{eee}) and (\ref{fff}) that
\begin{equation}
e^{C\sqrt{\lambda}}(\frac{r}{\lambda})^{\frac{-n}{2}}\|\hat{u}\|_{L^2(\mathbb
B_{\frac{5r}{3\lambda}})}\geq \|u\|_{L^\infty(\mathbb
B_{\frac{4r}{3\lambda}}\cap\partial\Omega)}.
\end{equation}
Taking (\ref{kin}) into consideration yields that
\begin{equation}
e^{C{\lambda}}(\frac{r}{\lambda})^{\frac{-n+1}{2}}\|{u}\|_{L^2(\mathbb
B_{\frac{r}{\lambda}}\cap\partial\Omega)}\geq
\|u\|_{L^\infty(\mathbb B_{\frac{4r}{3\lambda}}\cap\partial\Omega)}.
\end{equation}
Using H$\ddot{o}$lder's inequality,
\begin{equation}
e^{C{\lambda}}\|u\|_{L^\infty(\mathbb
B_{\frac{r}{\lambda}}\cap\partial\Omega)}\geq
\|u\|_{L^\infty(\mathbb B_{\frac{4r}{3\lambda}}\cap\partial\Omega)}.
\end{equation}
By iteration, we arrive at
\begin{equation}e^{C{\lambda}}\|u\|_{L^\infty(\mathbb
B_{\frac{r}{\lambda}}\cap\partial\Omega)}\geq
\|u\|_{L^\infty(\mathbb B_{\frac{2r}{\lambda}}\cap\partial\Omega)}.
\end{equation}
\end{proof}

With the help of the Carleman estimates and the analysis above, we
show Theorem 2 as below.
\begin{proof}[Proof of Theorem \ref{th2}]
As before, we assume any point $x_0\in \partial\Omega$ as origin.
Dropping the first term in the left hand side of (\ref{absoo}) gives
that
\begin{eqnarray}
\frac{1}{2}
e^{-2\beta\tilde{\phi}(\frac{R}{6})}R^{-n}\int_{A_{\frac{R}{12},
\frac{R}{6}}} |\hat{u}|^2 \leq  C(1+\lambda)^4
e^{-2\beta\tilde{\phi}(r/2)} {(\frac{r}{2})}^{-n}\int_{\mathbb
B_{r}}|\hat{u}|^2
\end{eqnarray}
for $0<r<\frac{r_0}{36}$. Let $$\mathcal {R}= \frac{1}{2}
e^{-2\beta\tilde{\phi}(\frac{R}{6})}R^{-n}\int_{A_{\frac{R}{12},
\frac{R}{6}}} |\hat{u}|^2.$$ Then $$(1+\lambda)^{-4}e^{2\beta(\log
\frac{r}{2}+\log(\log \frac{r}{2})^2)} \mathcal {R}\leq
r^{-n}\int_{\mathbb B_{r}}|\hat{u}|^2.$$ Since $\log\big((\log
\frac{r}{2})^2\big)/2>0$ as $r$ is small, we have
$$ C(1+\lambda)^{-4}r^{2\beta} \mathcal
{R} \leq r^{-n}\int_{\mathbb B_{r}}|\hat{u}|^2.$$ From
$(\ref{rkao})$ and $(\ref{who})$, we also know that $$\beta\leq
C(1+\lambda)\leq C\lambda.$$ Furthermore, we get
$$ Ce^{-C\lambda} r^{C\lambda} \mathcal
{R}\leq r^{-n}\int_{\mathbb B_{r}}|\hat{u}|^2 .$$ Since $\lambda>1$,
it is also true that
\begin{equation}
 e^{-C\lambda} {(\frac{r}{\lambda})}^{C\lambda} \mathcal
{R}\leq {(\frac{r}{\lambda})}^{-n}\int_{\mathbb
B_{{\frac{r}{\lambda}}}}|\hat{u}|^2.
\end{equation}
From (\ref{kin}), we obtain
\begin{equation}
(\frac{r}{\lambda})^{-n+1}\int_{\mathbb
B_\frac{r}{\lambda}\cap\partial\Omega} u^2\geq
\mathcal{R}e^{-C\lambda} {(\frac{r}{\lambda})}^{C\lambda}.
\end{equation}
Therefore,
$$ \|u\|_{L^\infty({\mathbb
B_\frac{r}{\lambda}\cap\partial\Omega})}\geq
\mathcal{R}e^{-C\lambda} {(\frac{r}{\lambda})}^{C\lambda}$$ which
implies the conclusion of Theorem 2.
\end{proof}

{\bf Remark:} After having completed the paper, we learn from arXiv
that A. R$\ddot{u}$land considered some similar problems for
fractional Schr$\ddot{o}$dinger equations \cite{R} using a different
method.

\section{Acknowledgement}
The author is indebted to Professor C.D. Sogge for guiding me into
the area of eigenfunctions and constant support.

\end{document}